\pdfoutput=1  
\documentclass[11pt]{amsart}
\usepackage{amssymb}
\usepackage[all]{xy} 
\usepackage{tikz}
\usepackage{tikz-cd}
\usepackage{tgtermes}
\usepackage[italic]{mathastext}

\title[Sheaves of Commutative DG Rings]{The Derived Category of Sheaves of 
Commutative DG Rings (Preview)}
\date{15 August 2016}
\usepackage{hyperref}
\hypersetup{colorlinks=false}

\author{Amnon Yekutieli}
\address{Department of  Mathematics,
Ben Gurion University, Be'er Sheva 84105, Israel}
\email{amyekut@math.bgu.ac.il}

\newtheorem{thm}[equation]{Theorem}
\newtheorem{cor}[equation]{Corollary}
\newtheorem{prop}[equation]{Proposition}

\theoremstyle{definition}
\newtheorem{dfn}[equation]{Definition}
\newtheorem{rem}[equation]{Remark}

\numberwithin{equation}{section}
\newcommand{\opn}{\operatorname}
\newcommand{\cat}[1]{\operatorname{\mathsf{#1}}}
\newcommand{\catt}[1]{{\operatorname{\mathsf{#1}}}} 

\newcommand{\cd}{\,{\cdot}\,}

\newcommand{\rmitem}[1]{\item[\text{\textup{(#1)}}]}

\newcommand{\mcal}[1]{\mathcal{#1}}

\newcommand{\mrm}[1]{\mathrm{#1}}

\newcommand{\OO}{\mcal{O}}
\newcommand{\MM}{\mcal{M}}
\newcommand{\NN}{\mcal{N}}

\renewcommand{\AA}{\mcal{A}}
\newcommand{\BB}{\mcal{B}}
\newcommand{\EE}{\mcal{E}}
\newcommand{\CC}{\mcal{C}}

\newcommand{\Ga}{\Gamma}
\newcommand{\si}{\sigma}

\newcommand{\ep}{\epsilon}

\newcommand{\Om}{\Omega}

\newcommand{\K}{\mathbb{K}}

\newcommand{\Z}{\mathbb{Z}}

\newcommand{\tup}[1]{\textup{#1}}

\newcommand{\ot}{\otimes}

\newcommand{\til}[1]{\tilde{#1}}

\newcommand{\bra}[1]{\langle #1 \rangle}
\renewcommand{\d}{\mathrm{d}}

\newcommand{\lb}{\linebreak}

\newcommand{\sub}{\subseteq}
\begin{document}

\begin{abstract}
In this short paper we outline (mostly without proofs) our new approach to the 
derived category of sheaves of commutative DG rings. The proofs will appear in 
a subsequent paper.

Among other things, we explain how to form the derived intersection of two 
closed subschemes inside a given algebraic scheme $X$, without recourse to 
simplicial or higher homotopical methods, and without any global assumptions on 
$X$.   
\end{abstract}

\maketitle

\tableofcontents

\setcounter{section}{-1}
\section{Introduction}

In the theory of {\em DG schemes} -- a simplified variant of {\em derived 
algebraic geometry} -- it is important to have a suitable way to resolve a 
sheaf of rings by a {\em flat sheaf of DG rings}. A typical problem is this: 
$X$ is a scheme, and $Y_1, Y_2 \subseteq X$ are closed subschemes. The {\em 
derived intersection} of $Y_1$ and $Y_2$ is a DG scheme  
\[ (Y, \OO_Y) =Y_1 \times_X^{\mrm{R}} Y_2 \]
whose underlying topological space is  $Y = Y_1 \cap Y_2$, and the structure 
sheaf
\[ \OO_Y = \OO_{Y_1} \ot_{\OO_X}^{\mrm{L}} \OO_{Y_2} \]
is a suitable sheaf of commutative DG rings on this space. 
If there exist flat resolutions $\phi_i : \AA_{i} \to \OO_{Y_i}$, by 
which we mean that $\AA_{i}$ is a flat commutative DG $\OO_X$-ring, 
and $\phi_i$ is a DG ring quasi-isomorphism, then we can take 
\[ \OO_{Y}  := (\AA_{1} \ot_{\OO_X} \AA_{2})|_Y . \]
(Of course, it is enough to resolve only one of the tensor factors.) 

In case $X$ is an affine scheme, or it is quasi-projective over a nice base 
ring $\K$, then it is quite easy to produce flat quasi-coherent DG ring 
resolutions of $\OO_{Y_i}$, and in this way to construct the sheaf of DG rings 
$\OO_Y$. This was already done in the paper \cite{CK} of Ciocan-Fontanine and 
Kapranov. 

But in general (for an arbitrary scheme $X$) there does not seem to be an 
existing method to obtain flat DG ring resolutions as sheaves on $X$ itself. 
Thus the derived intersection $(Y, \OO_Y)$ has until now existed only as an 
object in a much more complicated homotopical setting. See the preprint 
\cite{Be} of Behrend for one approach, and a survey of the 
approaches of To\"en et al.\ and of Lurie under ``derived stack'' in 
\cite{nLab}. Being only an object of a complicated homotopy category, the 
derived intersection $(Y, \OO_Y)$ is usually quite difficult to manipulate  
geometrically, and to form associated structures, such a derived 
module category over $\OO_{Y}$, etc. 

The first main innovation in this paper is the use of {\em commutative 
pseudo-semi-free sheaves of DG rings}. These sheaves enable the formation of 
flat commutative DG $\OO_X$-ring resolutions in great generality. But before 
saying what these sheaves are, we must present a few background concepts. 

For an open set $U \sub X$, let us denote by 
$\OO_{U \sub X}$ the extension by zero to $X$ of the sheaf $\OO_{U}$.
Suppose $I = \coprod_{n \leq 0} I^n$
is a graded set, and for each $i \in I$ we are given an open set 
$U_i \sub X$. Define 
\[ \EE^n := \bigoplus\nolimits_{i \in I^n} \, \OO_{U_i \sub X} \]
and 
\[ \EE := \bigoplus\nolimits_{n \leq 0} \, \EE^n . \] 
We call $\EE$ the {\em pseudo-free graded $\OO_X$-module} indexed by $I$. 
The summand $\EE^n$ is in degree $n$. 
Warning: the $\OO_X$-module $\EE$ is usually not quasi-coherent! 

The commutative tensor ring of $\EE$ over 
$\OO_X$ is called a {\em commutative pseudo-free graded ring}, and we denote it 
by $\OO_X[I]$. See Section \ref{sec:pseudo-free} for details. 
It is useful to view $\OO_X[I]$ as a commutative 
pseudo-polynomial graded $\OO_X$-ring, in which the elements 
$t_i := 1\in \Ga(U_i, \OO_X)$ play the role of variables (and we call them 
pseudo-generators). For each point $x \in X$ the stalk 
$\OO_X[I]_x$ is a genuine commutative polynomial graded $\OO_{X, x}$-ring,
in variables indexed by the graded set  
$\{ i \in I \mid x \in U_i \}$.

A sheaf of commutative DG $\OO_X$-rings $\til{\AA}$ is called pseudo-semi-free 
if the graded sheaf of rings $\til{\AA}^{\natural}$, that is gotten by 
forgetting the differential on $\til{\AA}$, is a commutative pseudo-free graded 
$\OO_X$-ring. As a DG $\OO_X$-module, such $\til{\AA}$ is K-flat. 
 
Any sheaf of commutative DG $\OO_X$-rings $\AA$ admits a 
commutative pseudo-semi-free DG $\OO_X$-ring resolution
$\til{\AA} \to \AA$. This is Theorem \ref{thm:140}.
Furthermore, these resolutions are unique up to suitable 
homotopies, that we explain below. 

This brings us to the second main innovation of this paper. 
Consider the category 
$\catt{DGR}^{\leq 0}_{\mrm{sc}} / \OO_X$
of commutative DG $\OO_X$-rings. We introduce a relation that we call {\em 
relative quasi-homotopy} on the set of morphisms in this category.
See Definition \ref{dfn:233}. This is a 
congruence, and hence there is the {\em homotopy category}, in the genuine 
sense, that we denote by $\cat{K}(\catt{DGR}^{\leq 0}_{\mrm{sc}} / \OO_X)$.
Its objects are the same as those of 
$\catt{DGR}^{\leq 0}_{\mrm{sc}} / \OO_X$,
and its morphisms are the relative quasi-homotopy classes.

We can also form the abstract localization of 
$\catt{DGR}^{\leq 0}_{\mrm{sc}} / \OO_X$
with respect to the quasi-isomorphisms, and the result is the {\em derived 
category of commutative DG $\OO_X$-rings}, that we denote by 
$\cat{D}(\catt{DGR}^{\leq 0}_{\mrm{sc}} / \OO_X)$.
There is a commutative diagram of functors
\begin{equation} \label{eqn:237}
\UseTips \xymatrix @C=6ex @R=6ex {
\catt{DGR}^{\leq 0}_{\mrm{sc}} / \OO_X
\ar[d]_{\opn{P}}
\ar[dr]^{\opn{Q}}
\\
\cat{K}(\catt{DGR}^{\leq 0}_{\mrm{sc}} / \OO_X)
\ar[r]^{\bar{\opn{Q}}} 
&
\cat{D}(\catt{DGR}^{\leq 0}_{\mrm{sc}} / \OO_X)
}
\end{equation}
The functor $\bar{\opn{Q}}$ is a {\em right Ore localization} 
with respect to the set of quasi-iso\-morph\-isms, and it is also {\em 
faithful}. This gives us very tight control on the morphisms in the derived 
category. 

The commutative pseudo-semi-free DG rings have a certain lifting property
that makes everything work. See Theorem \ref{thm:141}. 
However, the commutative pseudo-semi-free DG rings have a 
built-in finiteness property (basically coming from the fact that only a finite 
intersection of open sets of $X$ is open), thus preventing them from being 
``cofibrant objects''. This seems to indicate that there is no Quillen model 
structure on 
$\catt{DGR}^{\leq 0}_{\mrm{sc}} / \OO_X$. 

We can now say how we solve the problem of derived intersection. 
For $i = 1, 2$ we view $\OO_{Y_i}$ as living in 
$\catt{DGR}^{\leq 0}_{\mrm{sc}} / \OO_X$, 
and we choose pseudo-semi-free resolutions 
$\AA_i \to \OO_{Y_i}$. Define the topological space 
$Y := Y_1 \cap Y_2 \sub X$
and the commutative DG $\OO_X$-ring 
\[  \OO_Y  := (\AA_1 \ot_{\OO_X} \AA_2)|_Y . \]
Then the derived intersection of $(Y_1, \OO_{Y_1})$ and $(Y_2, \OO_{Y_2})$
is the DG ringed space $(Y, \OO_Y)$. 
There is a canonical isomorphism
\[ \OO_Y \cong \OO_{Y_1} \ot_{\OO_X}^{\mrm{L}} \OO_{Y_2} \]
in $\cat{D}(\catt{DGR}^{\leq 0}_{\mrm{sc}} / \OO_X)$.
See Corollary \ref{cor:265} for details. Furthermore, on 
any affine open set $V \sub X$ there is a canonical isomorphism in
$\cat{D}(\catt{DGR}^{\leq 0}_{\mrm{sc}} / \OO_V)$
between $\OO_Y|_V$ and any quasi-coherent DG sheaf presentation of the derived 
intersection. This is explained in Sections \ref{sec:der-inters} and 
\ref{sec:alg-geom} of the paper.

When the scheme $X$ is quasi-projective, it is not hard to see that our 
approach is compatible with that of \cite{CK}. For more general $X$ we did not 
attempt a comparison, but it is almost certain that our approach is compatible 
with those of Behrend, To\"en and Lurie. 

It stands to reason that the method outlined in this paper should allow 
a clean construction of the {\em cotangent complex} of $X$;
see Remark \ref{rem:255}. Our method should also 
permit a geometric version of Shaul's {\em derived completion of DG 
rings} from \cite{Sh}, but we do not have a formulation of it yet. 

Presumably our work can be extended without too much difficulty to sites that 
are more general than the Zariski topology of a scheme (e.g.\ to algebraic 
spaces, and maybe even to algebraic stacks). We leave this exploration aside for 
the time being. 

In the body of the paper (before Section \ref{sec:alg-geom}) we do not work with 
schemes, but rather with {\em topological spaces} in general. Indeed, 
the natural geometric object to which our approach applies is a {\em 
commutative DG ringed space}, which is a pair $(X, \AA)$, consisting of a 
topological space $X$ and a sheaf of commutative DG rings $\AA$ on it. Moreover, 
there is no need for a special base ring, such as a field of characteristic $0$; 
our constructions are valid over $\Z$. 

The present paper is only a preview, meant to convey our new ideas on this 
subject. Only a few proofs are given here (and some of them are just partial 
proofs). Our most important result is Theorem \ref{thm:140} on the existence of 
pseudo-semi-free resolutions, and for that we provide a sketch of a proof (the 
beginning of a full proof, and an indication how to complete it). Understanding 
the geometric principle of the proof of Theorem \ref{thm:140}, and 
combining it with the proofs of several algebraic results in \cite{Ye2}, 
should presumably allow experts to write their own proofs of the rest of the 
theorems in this paper. At any rate, we intend to publish a complete account of 
our approach in the future. Until then, we welcome feedback from readers, with 
suggestions of proofs, of better results, and also of counterexamples and 
refutations, in case there should be any...

\medskip \noindent
{\bf Acknowledgments.} 
I wish to thank Liran Shaul, Rishy Vyas, Vladimir Hinich and Donald Stanley for 
discussions.

\section{Pseudo-Free Sheaves of Commutative Graded Rings}
\label{sec:pseudo-free}

Let us fix a nonzero commutative base ring $\K$. For instance, $\K$ could be 
a field of characteristic $0$; or it could be the ring of integers $\Z$. 
Let $X$ be a topological space. The constant sheaf on $X$ with values in $\K$ 
is $\K_X$. The category of $\K_X$-modules (i.e.\ sheaves of $\K$-modules on 
$X$) is $\cat{M}(\K_X) = \cat{Mod} \K_X$.

\begin{dfn} \label{dfn:240}
A {\em sheaf of commutative graded $\K_X$-rings} is a sheaf of graded 
rings $\AA = \bigoplus_{m \leq 0} \AA^m$, together with a homomorphism 
$\K_X \to \AA^0$, that $\AA$ has the strong commutativity 
property: any local sections $a \in \AA^m$ and $b \in \AA^n$ satisfy 
$b \cd a = (-1)^{m n} \cd a \cd b$, and $a \cd a = 0$ if $m$ is odd.
\end{dfn}

The category of sheaves of commutative graded $\K_X$-rings is denoted by \lb
$\catt{GR}^{\leq 0}_{\mrm{sc}} / \K_X$. 

Suppose $U \sub X$ is an open set, with inclusion morphism
$g : U \to X$. For any $\K_U$-module $\NN$ its extension by zero to $X$ is the 
$\K_X$-module $g_!(\NN)$; and for any $\K_X$-module $\MM$ its restriction  
to $U$ is the $\K_U$-module $g^{-1}(\MM) = \MM|_U$. These operations are 
adjoint: there is a canonical isomorphism 
\begin{equation} \label{eqn:221}
\opn{Hom}_{\K_X} \bigl( g_!(\NN), \MM \bigr) \cong 
\opn{Hom}_{\K_U} \bigl( \NN, g^{-1}(\MM) \bigr)
\end{equation}
in $\cat{M}(\K)$. See \cite[Section II.1]{Ha}. 

\begin{dfn} \label{dfn:220}
Let $U \sub X$ be an open set, with inclusion morphism $g : U \to X$.
The {\em pseudo-free $\K_X$-module of pseudo-rank $1$ and pseudo-support $U$} 
is the $\K_X$-module  
\[ \K_{U \sub X} := g_!(\K_U) . \]
The element 
\[ t_U := 1 \in \Ga(U, \K_{U \sub X}) \cong \Ga(U, \K_X) \]
is called the {\em pseudo-free generator} of the $\K_X$-module $\K_{U \sub X}$.
\end{dfn}

Taking $\NN = \K_{U}$ in formula  (\ref{eqn:221}), 
we have canonical isomorphisms 
\begin{equation} \label{eqn:223}
\opn{Hom}_{\K_X} (\K_{U \sub X}, \MM) \cong 
\opn{Hom}_{\K_U} \bigl( \K_{U},  g^{-1}(\MM) \bigr) \cong \Ga(U, \MM) 
\end{equation}
in $\cat{M}(\K)$. The pseudo-free generator $t_U$ can be interpreted  
as a homomorphism 
$t_U : \K_{U \sub X} \to \K_X$ in $\cat{M}(\K_X)$.
This is actually an injective homomorphism, and thus we can view 
$\K_{U \sub X}$ as an {\em ideal sheaf} in $\K_X$. It is the ideal sheaf 
``pseudo-generated'' by $t_U$.

For a point $x \in X$ we denote by $\K_{U \sub X, x}$ the stalk 
of the sheaf $\K_{U \sub X}$ at $x$. 
If $x \in U$, then the stalk $\K_{U \sub X, x}$ is a free $\K$-module of rank 
$1$ with basis $t_U$. But if $x \notin U$ then $\K_{U \sub X, x} = 0$. 
If $U' \sub X$ is another open set, then 
\[ \K_{U \sub X} \ot_{\K_X} \K_{U' \sub X} \cong \K_{U \cap U' \sub X}  \]
canonically as $\K_X$-modules. The pseudo-free generators multiply:
\[ t_U \ot t_{U'} \mapsto t_{U \cap U'} . \]

The pseudo-free sheaves were used to great effect by Grothendieck in 
\cite[Section II.7]{RD}. See also our paper \cite[Section 3]{Ye1}, where the 
name ``pseudo-free'' was first used.  

\begin{dfn} \label{dfn:210}
A {\em generator specification} on $X$ is a triple 
\[ \bigl( I, \{ U_i \}_{i \in I}, \{ n_i \}_{i \in I} \bigr) \]
consisting of a set $I$, a collection $\{ U_i \}_{i \in I}$ of open sets of 
$X$, and a collection $\{ n_i \}_{i \in I}$ of nonpositive integers $n_i$. 
\end{dfn}

We often refer to the generator specification just as $I$, leaving the rest of 
the ingredients implicit. 
The numbers $n_i$ are called the cohomological degrees. 
For every index $i$ there is the the pseudo-free generator 
\begin{equation} \label{eqn:235}
t_i := t_{U_i} \in \Ga(U_i, \K_{U_i \sub X}) .
\end{equation}

Suppose a generator specification $I$ is given. For any $n$ let 
\begin{equation} \label{eqn:241}
I^n := \{ i \in I \mid n_i = n \} .
\end{equation}
Thus $I = \coprod_n I_n$, so it is a graded set. For any $x \in X$ we let
\begin{equation} \label{eqn:243}
I_x := \{ i \in I \mid x \in U_i \} .
\end{equation}

\begin{dfn} \label{dfn:241}
Given a generator specification $I$, the {\em graded pseudo-free  
$\K_X$-mod\-ule pseudo-generated by $I$} is 
\[ \EE := \bigoplus_{n \leq 0} \, \EE^n , \]
where for every $n$ the graded component of cohomological degree $n$ is  
\[ \EE^n := \bigoplus_{i \in I^n} \, \K_{U_i \sub X} . \]
\end{dfn}

Note that for any point $x \in X$ the stalk $\EE_x$ is a graded free
$\K$-module, with basis indexed by the graded set $I_x$.  

\begin{dfn} \label{dfn:242}
The {\em noncommutative pseudo-free graded $\K_X$-ring} pseudo-ge\-ne\-rated by 
 $I$ is \[ \K_X \bra{I} := \bigoplus_{l \geq 0} \, 
\EE \ot_{\K_X} \cdots \ot_{\K_X} \EE , \]
where in the $l$-th summand there are $l$ tensor factors.
The multiplication is the tensor product. 
\end{dfn}

Note that $\K_X \bra{I}$ is actually bigraded: it has the tensor 
grading $l$ and the cohomological grading $n$; but we are only interested in 
the cohomological grading. 

\begin{dfn} \label{dfn:211}
Let $I$ be a generator specification. 
The {\em commutative pseudo-free graded $\K_X$-ring} pseudo-generated by $I$ 
is the quotient $\K_X[I]$ of the $\K_X$-ring $\K_X \bra{I}$, modulo the 
two-sided ideal sheaf pseudo-generated by the local sections 
\[ t_{i} \cd t_{j} - (-1)^{n_i \cd n_j} \cd t_{j} \cd t_{i} \]
for all $i, j \in I$, and by the local sections
$t_{i} \cd t_{i}$ for all $i$ such that $n_i$ is odd. 
\end{dfn}

The homogeneous component of $\K_X[I]$ of cohomological degree $n$ is denoted 
by $\K_X[I]^n$. Thus
\begin{equation} \label{eqn:244}
\K_X[I] = \bigoplus_{n \leq 0} \, \K_X[I]^n . 
\end{equation}

\begin{prop} \label{prop:240}
For any point $x \in X$ there is a canonical graded $\K$-ring isomorphism 
\[ \K_X[I]_x \cong \K[ I_x] , \]
where $\K[ I_x]$ is the commutative graded polynomial ring on the 
collection of graded variables indexed by the graded set $I_x$. 
\end{prop}

\begin{cor} \label{cor:240}
For each $n$ the sheaf $\K_X[I]^n$ is flat over $\K_X$.
\end{cor}

\begin{prop} \label{prop:235}
Let $I$ be a generator specification, let 
$\AA \in \catt{GR}^{\leq 0}_{\mrm{sc}} / \K_X$, and for every $i \in I$
let $a_i \in \Ga(U_i, \AA^{n_i})$. Then there is a unique homomorphism 
$\phi : \K_X[I] \to \AA$ in 
$\catt{GR}^{\leq 0}_{\mrm{sc}} / \K_X$
such that $\phi(t_i) = a_i$. 
\end{prop}

\begin{proof}
Let $\EE$ be the pseudo-free $\K_X$-module pseudo-generated by $I$. 
The adjunction property (\ref{eqn:223})
says that there is a unique homomorphism 
$\phi : \EE \to \AA$ in $\cat{M}(\K_X)$ such that 
$\phi(t_i) = a_i$ on $U_i$. This extends uniquely to a homomorphism of 
$\K_X$-rings $\phi : \K_X \bra{I} \to \AA$
by the universal property of the tensor product. To be explicit, any 
pseudo-monomial 
\[ t_{i_1} \ot \cdots \ot t_{i_l} \in
\Ga \bigl( U_{i_1} \cap \cdots \cap U_{i_l}, \K_X \bra{I} \bigr) \]
goes to the element 
\[ a_{i_1}  \cdots a_{i_l} \in
\Ga \bigl( U_{i_1} \cap \cdots \cap U_{i_l}, \AA \bigr) . \]
Because $\AA$ is commutative, the two-sided ideal of graded commutators  
goes to zero, and therefore there is an induced homomorphism 
$\phi : \K_X[I] \to \AA$. The uniqueness is clear. 
\end{proof}

\begin{prop} \label{prop:236}
Let $I$ be a generator specification, let 
$\AA \in  \catt{GR}^{\leq 0}_{\mrm{sc}} / \K_X$, 
and define $\BB := \AA \ot_{\K_X} \K_X[I]$. 
Suppose for every $i \in I$ we are given an element
 $b_i \in \Ga(U_i, \BB^{n_i + 1})$. 
\begin{enumerate}
\item There is a unique derivation 
$\d : \BB \to \BB$ of degree $+1$ 
that extends the differential of $\AA$ and such that $\d(t_i) = b_i$. 

\item If $\d(b_i) = 0$ for all $i$, then $\d \circ \d = 0$. 
\end{enumerate}
\end{prop}

\begin{proof}
Like the previous proof, combined with \cite[Lemma 3.20]{Ye2}. 
\end{proof}

\section{Pseudo-Semi-Free Sheaves of Commutative DG Rings}

Again $X$ is a topological space. 

\begin{dfn} \label{dfn:245}
A {\em commutative DG $\K_X$-ring} is a 
sheaf of commutative graded $\K_X$-rings $\AA = \bigoplus_{i \leq 0} \AA^i$, as 
in Definition \ref{dfn:240}, with a $\K_X$-linear differential $\d$ of degree 
$+1$ that satisfies the graded Leibniz rule.

A homomorphism of DG $\K_X$-rings $\phi : \AA \to \BB$
is a homomorphism of sheaves that respects the DG $\K_X$-ring structure.
The category of commutative DG $\K_X$-rings is denoted by 
$\catt{DGR}^{\leq 0}_{\mrm{sc}} / \K_X$. 
\end{dfn}

In more conventional language, a commutative DG $\K_X$-ring $\AA$ would be 
called a sheaf of unital associative commutative nonpositive cochain 
differential graded $\K$-algebras on $X$. 

Commutative rings are viewed as DG rings concentrated in degree $0$.

\begin{dfn} \label{dfn:250}
A {\em commutative DG ringed space} over $\K$ is a pair $(X, \AA)$, 
where $X$ is a topological space, and $\AA$ commutative DG $\K_X$-ring.
\end{dfn}

\begin{dfn} \label{dfn:246}
Let $(X, \AA)$ be a commutative DG ringed space over $\K$.
A {\em commutative DG $\AA$-ring} is a pair $(\BB, \phi)$, where 
$\BB \in \catt{DGR}^{\leq 0}_{\mrm{sc}} / \K_X$,
and $\phi : \AA \to \BB$ is a homomorphism in 
$\catt{DGR}^{\leq 0}_{\mrm{sc}} / \K_X$. The morphisms between 
commutative DG $\AA$-rings are the obvious ones. 
The resulting category is denoted by
$\catt{DGR}^{\leq 0}_{\mrm{sc}} / \AA$.
\end{dfn}

To any $\AA \in \catt{DGR}^{\leq 0}_{\mrm{sc}} / \K_X$
we can assign its cohomology 
$\opn{H}(\AA)$, which is a graded $\K_X$-ring. Note that 
$\opn{H}(\AA)$ is the sheaf associated to the presheaf 
$U \mapsto \opn{H}(\Ga(U, \AA))$.
Cohomology is a functor 
\[ \opn{H} : \catt{DGR}^{\leq 0}_{\mrm{sc}} / \K_X \to 
\catt{GR}^{\leq 0}_{\mrm{sc}} / \K_X . \]
A homomorphism $\phi : \AA \to \BB$ is called a {\em quasi-isomorphism}
if $\opn{H}(\phi)$ is an isomorphism. 

Let $\AA$ be a commutative DG $\K_X$-ring. We denote by $\AA^{\natural}$ the 
graded $\K_X$-ring gotten by forgetting the differentials.
The commutative pseudo-free graded $\K_X$-ring pseudo-generated by a generator 
specification $I$ was introduced in Definition \ref{dfn:211}.

\begin{dfn}
Let $\phi : \AA \to \BB$ be a homomorphism
in $\catt{DGR}^{\leq 0}_{\mrm{sc}} / \K_X$.
We say that $\phi$ is a {\em pseudo-semi-free DG ring homomorphism}
in $\catt{DGR}^{\leq 0}_{\mrm{sc}} / \K_X$, and that 
$\BB$ is a {\em commutative pseudo-semi-free DG $\AA$-ring}, 
if there is an isomorphism 
\[ \BB^{\natural} \cong \AA^{\natural} \ot_{\K_X} \K_X[I] \]
of graded $\AA^{\natural}$-rings, for some generator 
specification $I$.
\end{dfn}

\begin{prop}  \label{prop:180}
Let $\BB$ be a pseudo-semi-free commutative DG $\AA$-ring. Then $\BB$
is K-flat as a DG $\AA$-module. 
\end{prop}

\begin{dfn} \label{dfn:181}
Let $(X, \AA)$ be a commutative DG ringed space, and let $\BB$ be a commutative 
DG $\AA$-ring. A {\em pseudo-semi-free commutative DG ring resolution} of $\BB$ 
over $\AA$, or a {\em pseudo-semi-free resolution of $\BB$ in 
$\catt{DGR}^{\leq 0}_{\mrm{sc}} / \AA$}, 
is a pair $(\til{\BB}, \psi)$, where $\til{\BB}$ is a pseudo-semi-free 
commutative DG $\AA$-ring, and $\psi : \til{\BB} \to \BB$ is a surjective 
quasi-isomorphism of DG $\AA$-rings. 
\end{dfn}

Here is the most important result of this paper. 

\begin{thm} \label{thm:140}
Let $(X, \AA)$ be a commutative DG ringed space, and let $\BB$ be a commutative 
DG $\AA$-ring. There exists a commutative pseudo-semi-free DG ring resolution of 
$\BB$ over $\AA$.
\end{thm}

\begin{proof}[Sketch of Proof]
This is a geometrization of the proof of \cite[Theorem 3.21(1)]{Ye2}, 
replacing variables by pseudo-generators. Instead of 
\cite[Lemmas 3.19 and 3.20]{Ye2}, here we use Propositions \ref{prop:235} and 
\ref{prop:236}. 

As in the proof of \cite[Theorem 3.21(1)]{Ye2},
we shall construct an ascending sequence 
$F_0(\til{\BB}) \sub F_1(\til{\BB}) \sub \cdots$
of pseudo-semi-free DG rings in 
$\catt{DGR}^{\leq 0}_{\mrm{sc}} / \AA$,
together with a compatible sequence of homomorphisms
$\phi_q : F_q(\til{\BB}) \to \BB$.
Moreover, there will be an  ascending sequence 
$F_0(I) \sub F_1(I) \sub \cdots$
of generator specifications, and compatible isomorphisms 
\[ F_q(\til{\BB})^{\natural} \cong \AA^{\natural} \ot_{\K_X} \K_X[F_q(I)] \]
of graded $\AA^{\natural}$-rings. 
The following conditions will be satisfied: 
\begin{enumerate}
\rmitem{i} The graded sheaf homomorphisms 
$\phi_q : F_q(\til{\BB}) \to \BB$,
$\opn{B}(\phi_q) : \opn{B}(F_q(\til{\BB})) \to \opn{B}(\BB)$ and
$\opn{H}(\phi_q) : \opn{H}(F_q(\til{\BB})) \to \opn{H}(\BB)$ 
are surjective in degrees $\geq -q$. 

\rmitem{ii} The graded sheaf homomorphism
$\opn{H}(\phi_q) : \opn{H}(F_q(\til{\BB})) \to \opn{H}(\BB)$ 
is bijective in degrees $\geq -q + 1$. 
\end{enumerate}
In condition (i), $\opn{B}(-)$ denotes coboundaries. 
The DG ring 
\[ \til{\BB} := \lim_{q \to} F_q(\til{\BB}) \]
and the homomorphism 
\[ \phi := \lim_{q \to} \phi_q : \til{\BB} \to  \BB \]
will have the desired properties. 

We shall just start the actual proof, for $q = 0$, imitating the proof of 
\cite[Theorem 3.21(1)]{Ye2}, and emphasizing the geometric considerations that 
arise here. 

For any point $x \in X$ we have the ring homomorphism $\AA^0_x \to \BB^0_x$. 
Let $\opn{B}^0(\BB^0_x)$ be the module of $0$-coboundaries. 
We choose a collection $\{ c''_k \}_{k \in K''_{0, x}}$ of elements in 
$\BB^{-1}_x$, indexed by a set $K''_{0 ,x}$, such that the collection  
$\{ \d(c''_k) \}_{k \in K''_{0, x}}$
generates $\opn{B}^0(\BB^0_x)$ as an $\AA^0_x$-module. 
Let $J''_{0, x}$ be another set, with a bijection 
$\d : K''_{0 ,x} \to J''_{0 , x}$. Define 
$b''_{j} := \d(c''_k) \in \BB^0_x$ for any $k \in K''_{0, x}$ and 
$j = \d(k) \in J''_{0 , x}$, 
so we have a collection 
$\{ b''_j \}_{j \in J''_{0, x}}$ of elements of $\BB^{0}_x$.

Next choose a collection 
$\{ b'_j \}_{j \in J'_{0, x}}$ of elements in $\BB^{0}_x$, 
indexed by a set $J'_{0 , x}$, such that the collection 
$\{ b''_j \}_{j \in J''_0} \cup \{ b'_j \}_{j \in J'_0}$
generates $\BB^0_x$ as an $\AA^0_x$-ring.

The indexing sets $J'_{0, x}$, $J''_{0, x}$ and $K''_{0, x}$ here correspond, 
respectively, to the indexing sets $Y'_0$,  $Y''_{0}$ and $Z''_{0}$ in the 
proof of \cite[Theorem 3.21(1)]{Ye2}. Indeed, they would be the same if 
$X = \{ x \}$, a space with a single point in it. 

For any index $k \in K''_{0 ,x}$ there is an open neighborhood 
$U''_{k}$ of $x$ such the element $c''_k \in \BB^{-1}_x$ extends to an element 
$c''_k \in \Ga(U''_k, \BB^{-1})$. 
This choice also gives us an element 
\[ b''_j := \d(c''_k) \in \Ga(U''_k, \BB^0)  \]
for $j = \d(k) \in J''_{0, x}$. 
Likewise, for any index $j \in J'_{0 ,x}$ there is an open neighborhood 
$U'_{j}$ of $x$ such the element $b'_j \in \BB^{0}_x$ extends to 
an element $b'_j \in \Ga(U'_j, \BB^{0})$. 

Define the set 
\[ F_0(I_x) := J'_{0, x} \sqcup J''_{0, x} \sqcup K''_{0, x} . \]
For $i \in F_0(I_x)$ define the open set $U_i := U''_{k}$ 
if either $i = k \in K''_{0, x}$ or $i = \d(k) \in J''_{0, x}$; 
and define $U_i := U'_{j}$ if $i = j \in J'_{0, x}$.
Define the integer $n_i := -1$ if $i = k \in K''_{0, x}$; and 
define $n_i := 0$ if $i \in J''_{0, x}$ or $i \in J'_{0, x}$.
Thus we have a generator specification 
\begin{equation} \label{eqn:260}
\bigl( F_0(I_x), \{ U_i \}_{i \in F_0(I_x)}, \{ n_i \}_{i \in F_0(I_x)} 
\bigr)
\end{equation}
``around $x$''.

Taking the union of (\ref{eqn:260}) over all points $x \in X$ we get a 
``global'' generator specification 
\[ \bigl( F_0(I), \{ U_i \}_{i \in F_0(I)}, \{ n_i \}_{i \in F_0(I)} 
\bigr) . \]
Define the DG ring 
\[ F_0(\til{\BB}) :=  \AA \ot_{\K_X} \K_X[F_0(I)] \]
with differential $\d(t_k) := t_{\d(k)}$
for any index $k \in K''_{0, x} \sub F_0(I_x) \sub F_0(I)$. This is possible by 
Proposition \ref{prop:236}. According to Proposition \ref{prop:235} there is
a homomorphism 
$\phi_0 : F_0(\til{\BB}) \to \BB$
of DG $\AA$-rings. Checking at stalks we see that $\phi_0$ satisfies condition 
(i) above for $q = 0$; condition (ii) is trivial for $q = 0$.  

At this stage we can shrink the indexing set $F_0(I)$, as long as condition 
(i) holds for $q = 0$. See Remark \ref{rem:260} regarding the possibility 
to make the set $F_0(I)$ finite. 

From here on the construction of $F_q(\til{\BB})$ for $q \geq 1$ continues 
along the lines of the proof of \cite[Theorem 3.21(1)]{Ye2}, with geometric 
arguments very similar to those we used above: for every $q$ we go to 
stalks at points, choose elements, and extend them to  open sets. 
\end{proof}

\begin{dfn}  \label{dfn:230}
Let $\eta : \AA \to \AA^+$ be a homomorphism in 
$\catt{DGR}^{\leq 0}_{\mrm{sc}} / \K_X$.
We say that $\eta$ is a {\em split acyclic pseudo-semi-free homomorphism} if
$\eta$ is pseudo-semi-free and a quasi-isomorphism, and there is a homomorphism 
$\ep : \AA^+ \to \AA$ in $\catt{DGR}^{\leq 0}_{\mrm{sc}} / \K_X$
such that $\ep \circ \eta = \opn{id}_{\AA}$.
\end{dfn}

\begin{thm} \label{thm:215}
Suppose $\phi : \AA \to \BB$ is a quasi-isomorphism in 
$\catt{DGR}^{\leq 0}_{\mrm{sc}} / \K_X$.
Then $\phi$ can be factored as 
$\phi = \phi^+ \circ \eta$,
where 
$\phi^+ : \AA^{+} \to \BB$ is a surjective quasi-isomorphism, 
and 
$\eta : \AA \to \AA^{+}$
is a split acyclic pseudo-semi-free homomorphism. 
\end{thm}

In a commutative diagram: 
\[ \UseTips \xymatrix @C=8ex @R=8ex {
&
\AA^+ 
\ar@{->>}[dr]^{\phi^+}
\ar@{->>}[dl]_{\ep}
\\
\AA
&
\AA
\ar@{>->}[u]_(0.4){\eta}
\ar@{>->>}[l]_{\opn{id}}
\ar[r]^{\phi}
&
\BB
} \]

\begin{proof}[Sketch of Proof]
We actually prove more: there is a split contractible commutative  
pseudo-semi-free DG ring $\CC$, and a homomorphism 
$\CC \to \BB$, such that $\AA^+ = \lb \AA \ot_{\K_X} \CC$. 
\end{proof}

\begin{thm} \label{thm:141}
Let $\AA \in \catt{DGR}^{\leq 0}_{\mrm{sc}} / \K$,
let $\BB \in \catt{DGR}^{\leq 0}_{\mrm{sc}} / \AA$,
and for $i = 0, 1$ let 
$\phi_i : \til{\BB}_i \to \BB$ be quasi-isomorphisms 
in $\catt{DGR}^{\leq 0}_{\mrm{sc}} / \AA$. 
Then there exists a pseudo-semi-free DG ring
$\til{\BB}' \in \catt{DGR}^{\leq 0}_{\mrm{sc}} / \AA$,
together with quasi-isomorphisms 
$\psi_i : \til{\BB}' \to \til{\BB}_i$
in $\catt{DGR}^{\leq 0}_{\mrm{sc}} / \AA$,  such that 
$\phi_0 \circ \psi_0 = \phi_1 \circ \psi_1$. 
\end{thm}

The statement is shown in the next commutative diagram in the category  \lb 
$\catt{DGR}^{\leq 0}_{\mrm{sc}} / \AA$.

\[ \UseTips \xymatrix @C=4ex @R=4ex {
&
\til{\BB}'
\ar@{-->}[dl]_{\psi_0}
\ar@{-->}[dr]^{\psi_1}
\\
\til{\BB}_0
\ar@{->}[dr]_{\phi_0}
&
&
\til{\BB}_1
\ar@{->}[dl]^{\phi_1}
\\
&
\BB
} \]

\begin{proof}[Sketch of Proof]
This is similar to the proof of \cite[Theorem 3.22]{Ye2}.
But the \lb pseudo-semi-free DG ring $\til{\BB}'$ has to be tailored, in terms 
of the open sets involved, to the DG rings $\til{\BB}_0$ and $\til{\BB}_1$.  
\end{proof}

Of course, by induction, this can be extended to any {\em finite} number of 
quasi-iso\-morph\-isms $\phi_i : \til{\BB}_i \to \BB$. 

\begin{rem}  \label{rem:180}
The construction of the DG ring $\til{\BB}$ in Theorem \ref{thm:141}
involves refinement. There is finiteness built into it (since open sets allow 
only finite intersections). In general, a single $\til{\BB}'$ will not work for 
an infinite collection of quasi-isomorphisms
$\phi_i : \til{\BB}_i \to \BB$. 

This seems to indicate that there are no cofibrant objects in 
$\catt{DGR}^{\leq 0}_{\mrm{sc}} / \K_X$, and thus there is no Quillen model 
structure!
\end{rem}

\section{Relative Quasi-Homotopies and the Derived Category}
\label{sec:quasi-hom}

The next definition is a variant of the left homotopy from Quillen theory (see 
\cite{Ho}). The DG ring $\BB^+$ plays the role of a {\em cylinder object}. 

\begin{dfn} \label{dfn:232}
Let $\AA \in \catt{DGR}^{\leq 0}_{\mrm{sc}} / \K_X$, and let 
$\phi_0, \phi_1 : \BB \to \CC$ be homomorphisms in 
$\catt{DGR}^{\leq 0}_{\mrm{sc}} / \AA$. A {\em homotopy between $\phi_0$ and 
$\phi_1$ relative to $\AA$} is a commutative diagram 
\[ \UseTips \xymatrix @C=8ex @R=6ex {
\BB
&
\BB \ot_{\AA} \BB
\ar@{->>}[l]_(0.6){\mu}
\ar[r]^(0.6){\phi_0 \, \ot \, \phi_1}
\ar[d]_{\eta}
&
\CC
\\
&
\BB^{+}
\ar@{->>}[ul]^{\ep}
\ar[ur]_{\phi}
} \]
in $\catt{DGR}^{\leq 0}_{\mrm{sc}} / \AA$,
where $\mu$ is the multiplication homomorphism, and $\ep$ is a 
quasi-iso\-morph\-ism. If a homotopy exists, then we say that  $\phi_0$ and 
$\phi_1$ are {\em homotopic relative to $\AA$}.
\end{dfn}

\begin{dfn} \label{dfn:233}
Let $\phi_0, \phi_1 : \BB \to \CC$ be homomorphisms in 
$\catt{DGR}^{\leq 0}_{\mrm{sc}} / \AA$. The homomorphisms $\phi_0$ and 
$\phi_1$ are said to be {\em quasi-homotopic relative to $\AA$} is there is a 
quasi-isomorphism $\psi : \til{\BB} \to \BB$ in 
$\catt{DGR}^{\leq 0}_{\mrm{sc}} / \AA$
such that $\phi_0 \circ \psi$ and $\phi_1 \circ \psi$ are homotopic relative 
to $\AA$, in the sense of Definition \ref{dfn:232}. 
This  relation on morphisms in 
$\catt{DGR}^{\leq 0}_{\mrm{sc}} / \AA$ is called {\em relative quasi-homotopy}. 
\end{dfn}

\[ \UseTips \xymatrix @C=8ex @R=6ex {
\til{\BB}
\ar[r]^{\psi}
\ar@(ur,ul)[rr]^{\phi_i \, \circ \, \psi}
&
\BB
\ar[r]^{\phi_i}
&
\CC
} \]

\begin{thm} \label{thm:245}
Suppose $\til{\phi}_0, \til{\phi}_1 : \til{\BB} \to \til{\CC}$,
$\phi_0, \phi_1 : \til{\BB} \to \CC$ and $\si : \til{\CC} \to \CC$
are homomorphisms in $\catt{DGR}^{\leq 0}_{\mrm{sc}} / \AA$, 
such that  that 
$\phi_i = \si \circ \til{\phi}_i$, 
$\si$ is a quasi-isomorphism, and the homomorphisms 
$\phi_0$ and $\phi_1$ are homotopic relative to $\AA$. 
Then there is a pseudo-semi-free resolution 
$\til{\psi}: \til{\BB}' \to \til{\BB}$
in $\catt{DGR}^{\leq 0}_{\mrm{sc}} / \AA$,
such that 
$\til{\phi}_0 \circ \til{\psi}$ and $\til{\phi}_1  \circ \til{\psi}$ are 
homotopic relative to $\AA$. 
\end{thm}

Here are the commutative diagrams in 
$\catt{DGR}^{\leq 0}_{\mrm{sc}} / \AA$, for $i = 0, 1$~:
\[ \UseTips \xymatrix @C=8ex @R=6ex {
\til{\BB}'
\ar@{-->}[r]^{\til{\psi}}
\ar@{-->}[dr]_{\til{\phi_i} \, \circ \, \til{\psi}}
&
\til{\BB}
\ar[d]^{\til{\phi}_i}
\ar[dr]^{\phi_i}
\\
&
\til{\CC}
\ar[r]^{\si}
&
\CC
} \]

\begin{thm} \label{thm:232}
Let $\AA \in \catt{DGR}^{\leq 0}_{\mrm{sc}} / \K_X$.
The relation of relative quasi-homotopy is a congruence on the 
category $\catt{DGR}^{\leq 0}_{\mrm{sc}} / \AA$. 
\end{thm}

In order to have a visible distinction from the Quillen model category 
notation, below we choose notation that resembles the Grothendieck notation in 
\cite{RD}. 

\begin{dfn} \label{dfn:234}
Let $\AA \in \catt{DGR}^{\leq 0}_{\mrm{sc}} / \K_X$.
The {\em homotopy category} of $\catt{DGR}^{\leq 0}_{\mrm{sc}} / \AA$
is its quotient category modulo the relative quasi-homotopy congruence, and we 
denote it by $\cat{K}(\catt{DGR}^{\leq 0}_{\mrm{sc}} / \AA)$.
\end{dfn}

Thus for any pair of objects $\BB, \CC$ we have 
\[ \opn{Hom}_{\cat{K}(\catt{DGR}^{\leq 0}_{\mrm{sc}} / \AA)}(\BB, \CC) = 
\frac{\opn{Hom}_{\catt{DGR}^{\leq 0}_{\mrm{sc}} / \AA}(\BB, \CC)}
{\tup{relative quasi-homotopy}} . \]
There is a functor 
\[ \opn{P} :  \catt{DGR}^{\leq 0}_{\mrm{sc}} / \AA \to 
\cat{K}(\catt{DGR}^{\leq 0}_{\mrm{sc}} / \AA) \]
that is the identity on objects and surjective on morphisms. 

Within $\cat{K}(\catt{DGR}^{\leq 0}_{\mrm{sc}} / \AA)$
we have the set of quasi-isomorphisms, and they form a multiplicatively closed 
set of morphisms. 

\begin{thm} \label{thm:233}
The quasi-isomorphisms in 
$\cat{K}(\catt{DGR}^{\leq 0}_{\mrm{sc}} / \AA)$
satisfy the right Ore condition and the right cancellation condition. 
\end{thm}

\begin{dfn} \label{dfn:235}
Let $\AA \in \catt{DGR}^{\leq 0}_{\mrm{sc}} / \K_X$.
The {\em derived category} of $\catt{DGR}^{\leq 0}_{\mrm{sc}} / \AA$
is its localization with respect to the quasi-isomorphisms. We denote it by 
$\cat{D}(\catt{DGR}^{\leq 0}_{\mrm{sc}} / \AA)$.
\end{dfn}

By definition there is a  functor 
\[ \opn{Q} :  \catt{DGR}^{\leq 0}_{\mrm{sc}} / \AA \to 
\cat{D}(\catt{DGR}^{\leq 0}_{\mrm{sc}} / \AA) \]
that is the identity on objects. Since relatively quasi-homotopic morphisms in 
$\catt{DGR}^{\leq 0}_{\mrm{sc}} / \AA$
become equal in 
$\cat{D}(\catt{DGR}^{\leq 0}_{\mrm{sc}} / \AA)$,
we get a commutative diagram of functors
\[ \UseTips \xymatrix @C=6ex @R=6ex {
\catt{DGR}^{\leq 0}_{\mrm{sc}} / \AA
\ar[d]_{\opn{P}}
\ar[dr]^{\opn{Q}}
\\
\cat{K}(\catt{DGR}^{\leq 0}_{\mrm{sc}} / \AA)
\ar[r]^{\bar{\opn{Q}}} 
&
\cat{D}(\catt{DGR}^{\leq 0}_{\mrm{sc}} / \AA)
} \]

\begin{cor} \label{cor:231}
The functor 
\[ \bar{\opn{Q}} : \cat{K}(\catt{DGR}^{\leq 0}_{\mrm{sc}} / \AA) \to 
\cat{D}(\catt{DGR}^{\leq 0}_{\mrm{sc}} / \AA) \]
is a right Ore localization, and it is also faithful. 
\end{cor}

This tells us that any morphism in 
$\cat{D}(\catt{DGR}^{\leq 0}_{\mrm{sc}} / \AA)$ can be expressed as a simple 
right fraction:
\[ \opn{Q}(\phi) \circ \opn{Q}(\psi)^{-1}  \]
where $\phi, \psi$ are morphisms in $\catt{DGR}^{\leq 0}_{\mrm{sc}} / \AA$, 
and $\psi$ is a quasi-isomorphism. Moreover, there is equality  
\[ \opn{Q}(\phi_1) = \opn{Q}(\phi_2) \]
in $\cat{D}(\catt{DGR}^{\leq 0}_{\mrm{sc}} / \AA)$ iff  
$\phi_1$ and $\phi_2$ are relatively quasi-homotopic in
$\catt{DGR}^{\leq 0}_{\mrm{sc}} / \AA$.

We do not wish to perform a detailed study of maps between commutative DG 
ringed spaces in this paper. We only note that:

\begin{prop} \label{prop:250}
Let $V \sub X$ be an open set. The restriction functor 
\[ \catt{DGR}^{\leq 0}_{\mrm{sc}} / \AA \to 
\catt{DGR}^{\leq 0}_{\mrm{sc}} / \AA|_V , \quad  
\BB \mapsto \BB|_V \]
induces functors 
\[ \cat{K}(\catt{DGR}^{\leq 0}_{\mrm{sc}} / \AA) \to 
\cat{K}(\catt{DGR}^{\leq 0}_{\mrm{sc}} / \AA|_V) \]
and
\[ \cat{D}(\catt{DGR}^{\leq 0}_{\mrm{sc}} / \AA) \to 
\cat{D}(\catt{DGR}^{\leq 0}_{\mrm{sc}} / \AA|_V) , \]
that commute with the functors $\opn{Q}$,  $\opn{P}$ and $\bar{\opn{Q}}$.
\end{prop}

\section{Left Derived Tensor Products of Sheaves of DG Rings}
\label{sec:der-inters}

As before, $(X, \AA)$ is a  commutative DG ringed space over $\K$. 

\begin{thm} \label{thm:246}
Consider the commutative DG ringed space $(X, \AA)$. 
There is a bifunctor 
\[ (- \ot^{\mrm{L}}_{\AA} -) : 
\cat{D}(\catt{DGR}^{\leq 0}_{\mrm{sc}} / \AA) \times 
\cat{D}(\catt{DGR}^{\leq 0}_{\mrm{sc}} / \AA) \to 
\cat{D}(\catt{DGR}^{\leq 0}_{\mrm{sc}} / \AA) \, , \] 
together with a morphism
\[ \xi : \opn{Q} \circ \, (- \ot_{\AA} -)  \to (- \ot^{\mrm{L}}_{\AA} -)  \]
of bifunctors 
\[ \catt{DGR}^{\leq 0}_{\mrm{sc}} / \AA \, \times \,  
\catt{DGR}^{\leq 0}_{\mrm{sc}} / \AA \to 
\cat{D}(\catt{DGR}^{\leq 0}_{\mrm{sc}} / \AA) \, , \]
with this property\tup{:} if  
$\BB, \CC \in \catt{DGR}^{\leq 0}_{\mrm{sc}} / \AA$
are such that at least one of them is K-flat over $\AA$, then the morphism
\[ \xi_{\BB, \CC} : \BB \ot_{\AA} \CC \to 
\BB \ot^{\mrm{L}}_{\AA}  \CC \]
in 
$\cat{D}(\catt{DGR}^{\leq 0}_{\mrm{sc}} / \AA)$
is an isomorphism. 
\end{thm}

\begin{proof}[Sketch of Proof]
Given $\BB, \CC \in \catt{DGR}^{\leq 0}_{\mrm{sc}} / \AA$
we choose pseudo-semi-free resolutions 
$\til{\BB} \to \BB$ and $\til{\CC} \to \CC$
in $\catt{DGR}^{\leq 0}_{\mrm{sc}} / \AA$, and define 
\[ \BB \ot^{\mrm{L}}_{\AA} \CC := 
\til{\BB} \ot_{\AA} \til{\CC} . \]
The results of Section \ref{sec:quasi-hom} shows that this is a derived 
functor. 
\end{proof}

The derived tensor product respects localizations, as in Proposition 
\ref{prop:250}.

\section{Resolutions in Algebraic Geometry} \label{sec:alg-geom}

In this section $(X, \OO_X)$ is a scheme (over the base ring $\K$). 

Let $\AA$ be a quasi-coherent commutative DG $\OO_X$-ring; by this we mean 
that each $\OO_X$-module $\AA^p$ is quasi-coherent. 
Let $V \subseteq X$ be an affine open set, and write 
$C := \Ga(V, \OO_X)$ and $A := \Ga(V, \AA)$.  
So $A$ is a commutative DG $C$-ring. 
We can build a commutative semi-free DG $C$-ring resolution
$g : \til{A} \to A$. Each $\til{A}^p$ is a $C$-module, and we can 
sheafify it to get a quasi-coherent sheaf $\til{\AA}^p$ on $V$. 
In this way we obtain a {\em commutative semi-free DG $\OO_V$-ring resolution} 
$g : \til{\AA} \to \AA|_V$.

\begin{thm} \label{thm:142}
Let $(X, \OO_X)$ be a scheme, let $\AA$ be a quasi-coherent 
commutative DG $\OO_X$-ring, and let let $V \subseteq X$ be an affine open set.
Suppose we are given a commutative pseudo-semi-free DG $\OO_X$-ring 
resolution $g : \til{\AA} \to \AA$ on all of $X$, 
and also a commutative semi-free DG $\OO_V$-ring resolution 
$g' : \til{\AA}' \to \AA|_V$ on $V$. 
Then there exists a commutative pseudo-semi-free DG $\OO_V$-ring resolution 
$g'' : \til{\AA}'' \to \AA|_V$ on $V$, and DG $\OO_V$-ring quasi-isomorphisms 
$f : \til{\AA}'' \to \til{\AA}|_V$ and
$f' : \til{\AA}'' \to \til{\AA}'$, such that 
$g|_V \circ f =  g' \circ f' = g''$. 
\end{thm}

\begin{proof}
The commutative semi-free DG $\OO_V$-ring resolution 
$g'' : \til{\AA}'' \to \AA|_V$ is just a special case of a 
commutative pseudo-semi-free DG $\OO_V$-ring resolution; 
so we can apply Theorem \ref{thm:141} to it and to 
$g|_V : \til{\AA}|_V \to \AA|_V$.
\end{proof}

What Theorem \ref{thm:142} says is that locally our 
commutative pseudo-semi-free resolutions are the same as the quasi-coherent 
resolutions that were considered in \cite{CK}. 

In the next corollary we identify a sheaf on a closed subset $Y \sub X$ 
with its pushforward to $X$. 

\begin{cor} \label{cor:265}
Let $(Y_1, \OO_{Y_1})$ and $(Y_2, \OO_{Y_2})$ be closed subschemes of 
$(X, \OO_{X})$. There is a commutative DG ringed space
$(Y, \OO_{Y})$, such that 
\[ Y = Y_1 \cap Y_2 \sub X \]
as topological spaces, and 
\[ \OO_{Y} = \OO_{Y_1} \ot^{\mrm{L}}_{\OO_{X}} \OO_{Y_2} \]
in $\cat{D}(\catt{DGR}^{\leq 0}_{\mrm{sc}} / \OO_X)$. 
Thus 
\[ (Y, \OO_{Y}) = 
(Y_1, \OO_{Y_1}) \times^{\mrm{R}}_{(X, \OO_{X})} (Y_2, \OO_{Y_2}) , \]
the derived intersection of these subschemes.
\end{cor}

\begin{proof}
For $i = 1, 2$ we choose pseudo-semi-free resolutions 
$\AA_i \to \OO_{Y_i}$ in \lb 
 $\catt{DGR}^{\leq 0}_{\mrm{sc}} / \OO_X$. 
Let 
$Y := Y_1 \cap Y_2$, and define
\[ \OO_Y := (\AA_1 \ot_{\OO_X} \AA_2)|_Y , \]
the restriction of the DG $\OO_X$-ring $\AA_1 \ot_{\OO_X} \AA_2$
to the closed subset $Y$. A calculation in stalks shows that the canonical DG 
ring homomorphism
\[ \AA_1 \ot_{\OO_X} \AA_2 \to \OO_Y \]
is a quasi-isomorphism. 
\end{proof}

\begin{rem} \label{rem:255}
Here is a speculation regarding the {\em cotangent complex} of the scheme $X$. 
For this we view $\OO_X$ as living in 
$\catt{DGR}^{\leq 0}_{\mrm{sc}} / \K_X$,
where $\K$ is the base ring. Let $\AA \to \OO_X$ be a 
commutative pseudo-semi-free resolution in 
$\catt{DGR}^{\leq 0}_{\mrm{sc}} / \K_X$. 
There is a DG $\AA$-module $\Om^1_{\AA / \K}$, defined as the sheafification of 
the presheaf
\[ V \mapsto \Om^1_{\Ga(V, \AA) / \K} . \]
Let 
\[ \opn{L}_X := \OO_X \ot_{\AA} \Om^1_{\AA / \K} \in \cat{D}(\OO_X) \]
We believe that $\opn{L}_X$ is canonically isomorphic (in the derived category 
$\cat{D}(\OO_X)$) to the cotangent complex as constructed in  \cite{Il}. 
\end{rem}

\begin{rem} \label{rem:260}
If the scheme $X$ is noetherian, and if $\AA$ is a coherent commutative 
$\OO_X$-ring (e.g.\ $\AA = \OO_Y$ for a closed subscheme $Y \sub X$), then 
it is possible to find a commutative pseudo-semi-free resolution
$\til{\AA} \to \AA$
in $\catt{DGR}^{\leq 0}_{\mrm{sc}} / \OO_X$ such that 
$\til{\AA}^{\natural} \cong \OO_X \ot_{\K_X} \K_X[I]$, 
and the indexing set $I$ is finite in each degree. 
This is by the results of \cite[Section II.7]{RD}. 
\end{rem}


\end{document}